\documentclass[11pt]{amsart}

\usepackage{amssymb}
\usepackage{amsmath}
\usepackage{amsthm}
\usepackage{enumerate}
\usepackage{graphicx}
\usepackage{blkarray}
\usepackage{mathrsfs}

\usepackage{caption}
\usepackage{subcaption}

\usepackage[hidelinks]{hyperref}

\theoremstyle{plain}
\newtheorem{thm}{Theorem}[section]

\newtheorem{lem}[thm]{Lemma}
\newtheorem{cor}[thm]{Corollary}

\numberwithin{equation}{section}

\theoremstyle{definition}

\theoremstyle{remark}

\theoremstyle{plain}

\newcommand{\thmref}[1]{Theorem~\ref{#1}}

\newcommand{\lemref}[1]{Lemma~\ref{#1}}
\newcommand{\corref}[1]{Corollary~\ref{#1}}
\newcommand{\figref}[1]{Figure~\ref{#1}}

\newcommand{\eqnref}[1]{Equation~\eqref{#1}}

\newcommand{\calA}{{\mathcal A}}

\newcommand{\calH}{{\mathcal H}}

\newcommand{\calZ}{{\mathcal Z}}

\newcommand{\CC}{{\mathbb C}}
\newcommand{\DD}{{\mathbb D}}
\newcommand{\HH}{{\mathbb H}}
\newcommand{\RR}{{\mathbb R}}

\newcommand{\wtilde}{\widetilde}

\DeclareMathOperator{\Aff}{Aff}
\DeclareMathOperator{\LIP}{LIP}
\DeclareMathOperator{\LSP}{LSP}
\newcommand{\id}{\mathrm{id}}

\begin{document}

\title{The space of immersed polygons}
\author{Maxime Fortier Bourque}
\address{D\'epartement de math\'ematiques et de statistique, Universit\'e de Montr\'eal, 2920, chemin de la Tour, Montr\'eal (QC), H3T 1J4, Canada}
\email{maxime.fortier.bourque@umontreal.ca}

\begin{abstract}
We use the Schwarz--Christoffel formula to show that for every $n\geq 3$, the space of labelled immersed $n$-gons in the plane up to similarity is homeomorphic to $\RR^{2n-4}$. We then prove that all immersed triangles, quadrilaterals, and pentagons are embedded, from which it follows that the space of labelled simple $n$-gons up to similarity is homeomorphic to $\RR^{2n-4}$ if $n\in \{3,4,5\}$. This was first shown  by Gonz\'ales and L\'opez-L\'opez for $n=4$ and conjectured to be true for every $n\geq 5$ by González and Sedano-Mendoza.
\end{abstract}


\maketitle

\section{Introduction}

An \emph{immersed polygon} in the plane is a closed polygonal curve $\phi: S^1 \to \CC$ (considered up to orientation-preserving reparametrization) which extends to an orientation-preserving local homeomorphism $f: \overline{\DD} \to \CC$. An immersed polygon $\phi$ is \emph{labelled} if we choose distinct points $z_1, \ldots, z_n \in S^1$ in counter-clockwise order such that $\phi$ maps each component of $S^1 \setminus \{ z_1 , \ldots, z_n \}$ to the interior of a straight line segment. The images $w_j = \phi(z_j)$ (some of which may coincide) are the \emph{vertices} of the immersed polygon. At each vertex $w_j$, the \emph{interior angle} is the clockwise angle $\angle w_{j-1} w_j w_{j+1} \in (0,2\pi)$, where indices are taken modulo $n$.

\begin{figure}[htp]
\centering
\includegraphics[width=0.4\textwidth]{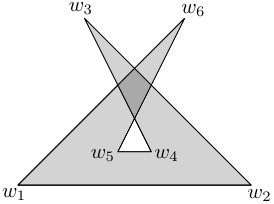}
\caption{An immersed hexagon.} \label{fig:hex}
\end{figure}

Since the vertices of a labelled immersed polygon determine the map $\phi$ up to reparametrization, the set $\LIP(n)$ of labelled immersed $n$-gons can be identified with a subset of $\CC^{n}$, which induces a topology on $\LIP(n)$. With this topology, the group 
\[\Aff(\CC)= \{ z \mapsto az+b : a \in \CC \setminus \{ 0\}, b \in \CC \}\]
of complex affine transformations (or, equivalently, of orientation-preserving similarities of $\CC$) acts on $\LIP(n)$ by homeomorphisms.

The goal of this note is to show that the quotient space $\LIP(n)/\Aff(\CC)$ is homeomorphic to $\RR^{2n-4}$.

\begin{thm}\label{thm:main}
For every $n\geq 3$, the quotient of the space $\LIP(n)$ of labelled immersed $n$-gons in the plane by the action of the group $\Aff(\CC)$ is homeomorphic to $\RR^{2n-4}$.
\end{thm}

The proof is based on the Schwarz--Christoffel formula in complex ana\-ly\-sis. It shows that $\LIP(n)/\Aff(\CC)$ is a trivial $\RR^{n-1}$-bundle (with fibers parametrized by the interior angles) over the Teichm\"uller space of closed disks with $n$ marked points on the boundary, itself homeomorphic to $\RR^{n-3}$.

We then prove that for $n \in \{3,4,5\}$, all immersed $n$-gons are simple (i.e., embedded), which has the following consequence.

\begin{cor} \label{cor:emb}
For every $n\in \{3,4,5\}$, the quotient of the space $\LSP(n)$ of simple $n$-gons in the plane (with vertices labelled in counter-clockwise order) by the action of the group $\Aff(\CC)$ is homeomorphic to $\RR^{2n-4}$.
\end{cor}

For $n=3$, this result is well-known (see e.g. \cite{Baez}) and elementary. Here is a proof: the map that sends a triangle with counter-clockwise vertices $w_1$, $w_2$, $w_3$ to \[\frac{w_3-w_1}{w_2-w_1}\] induces a homeomorphism from $\LSP(3)/\Aff(\CC)$ to the upper half-plane $\HH$, which is homeomorphic to $\RR^2$ via the map $x+iy \mapsto (x,\log y)$. For $n=4$, the result was first proved in \cite{GonzalezLopez} and then again in \cite{GSM} with different methods. In \cite{GSM}, González and Sedano-Mendoza conjectured that the result holds for every $n\geq 5$ as well. The above corollary settles the case $n=5$ and reproves the previous cases in a unified way. For general $n\geq 3$, it is known that $\LSP(n)/\Aff(\CC)$ is simply-connected \cite[Theorem 3.2]{Gonzalez}. 

\subsection*{Relation to other work}

Various other moduli spaces of planar polygons have been studied in the li\-te\-rature before, such as those of polygons with fixed side lengths (known as \emph{polygonal linkages}) \cite{Hausmann,KapovichMillson,Equilateral,convexifying,ShimamotoWooters} or with fixed interior angles \cite{BavardGhys,ShapesOfStars}. The former moduli spaces can be quite complicated topologically, while the latter can be parametrized by interiors of convex hyperbolic polyhedra (hence are cells) if all interior angles are in $(0,\pi)$.

The works \cite{convexifying} and \cite{ShimamotoWooters} in particular study simple polygons with fixed side lengths. Some of their results can be restated as follows. Given positive real numbers $\ell_1, \ldots, \ell_n$ such that $\ell_{j} < \sum_{k\neq j} \ell_k$ for every $j$, the subset $X(\ell)$ of polygons in $\LSP(n)/\Aff(\CC)$ whose vector of side lengths is a constant multiple of $\ell = (\ell_1, \ldots, \ell_n)$ is contractible \cite[Corollary 3]{convexifying}. If furthermore the $\ell_j$ cannot partitioned into two sets with the same sum, then $X(\ell)$ is homeomorphic to $\RR^{n-3}$ \cite[Theorem 7]{ShimamotoWooters}. If the additionnal hypothesis on subsets with equal sums could be removed in this result of Shimamoto--Wooters, then the conjecture of González and Sedano-Mendoza would follow, since the set of realizable vectors of side lengths up to scaling is homeomorphic to $\RR^{n-1}$, and a fiber bundle over a contractible base is always trivial.

This is all to say that the statement of \corref{cor:emb} is likely to be true for every $n\geq 3$ even though that cannot be deduced from \thmref{thm:main}.

\section{Schwarz--Christoffel tranformations}

We will work with the closed upper half-plane \[\overline{\HH} = \HH \cup \RR \cup \{ \infty \}\] instead of the closed unit disk. 

Let $n\geq 3$. Given prevertices $z_1 < z_2 < \cdots < z_{n-1}$ in $\RR$ and $z_n = \infty$, parameters $\alpha_1, \ldots, \alpha_n \in (0,2)$ such that
\begin{equation} \label{eqn:angles}
\sum_{j=1}^n \alpha_j = n-2,
\end{equation}
 and constants $A\in \CC\setminus\{0\}$, $B \in \CC$, the corresponding \emph{Schwarz--Christoffel transformation} $F: \overline{\HH} \to \CC$ is defined by the formula
\begin{equation}
F(z):= A \int_{[i,z]} \prod_{j=1}^{n-1}(\zeta - z_j)^{\alpha_j - 1} \, d\zeta + B,
\end{equation}
where $[i,z]$ denotes the straight line segment from the imaginary unit $i$ to $z$, and the powers $(\zeta - z_j)^{\alpha_j - 1}:= \exp((\alpha_j - 1)L(\zeta - z_j))$ are defined using the branch $L$ of the logarithm in $\overline{\HH}\setminus\{0\}$ that coincides with the usual logarithm on $(0,\infty)$. The main result concerning these maps can be summarized as follows \cite{Christoffel,Schwarz}.

\begin{thm}[Schwarz, Christoffel] \label{thm:SC}
For any choice of parameters as above, the Schwarz--Christof\-fel transformation $F$ is a local homeomorphism, holomorphic in $\HH$, whose restriction to $\RR \cup \{ \infty \}$ is an immersed $n$-gon in $\CC$ with interior angle $\alpha_j \pi$ at vertex $F(z_j)$ for every $j \in \{ 1, \ldots, n \}$. The resulting polygon is simple if and only if $F$ is injective. Conversely, any immersed $n$-gon in $\CC$ arises as the restriction to $\RR \cup \{ \infty \}$ of a Schwarz--Christoffel transformation, unique up to pre-composition by an affine map $z \mapsto az+b$ with $a>0$ and $b \in \RR$.
\end{thm}

While the last part of the theorem is usually stated for simple polygons, essentially the same proof works in the immersed case. The only minor difference is that one needs to use the uniformization theorem instead of the Riemann mapping theorem (the extension $f$ of an immersed polygon $\phi$ can be made holomorphic by pulling back the complex structure from the complex plane). Christoffel \cite{Christoffel2} even proved a version of the theorem for branched covers, in which case there are more terms in the integrand to account for the branch points. See \cite{DriscollTrefethen} and the references therein for various generalizations of the Schwarz--Christoffel formula.

We will need another well-known fact, namely, that the map $F$ and the resulting immersed polygon depend continuously on the parameters $z_1,\ldots,z_{n-1}$, $\alpha_1, \ldots, \alpha_n$, and $A$,$B$. This is because away from the prevertices, the integrand depends uniformly continuously on the parameters, and the contribution near the prevertices can be controlled locally uniformly.

We will say that a Schwarz--Christoffel transformation $F$ is \emph{normalized} if its prevertices satisfy $z_1 = -1$ and $z_2 = 0$ in addition to $z_n = \infty$, which determines $F$ uniquely for a given immersed polygon. It is known that the normalized  Schwarz--Christoffel transformation $F$ and its remaining para\-me\-ters depend continuously on the immersed polygon, just like normalized Riemann maps depend continuously on the domain. This can be shown u\-sing a normal families argument and the uniqueness of the normalized map.

Given an immersed polygon $\phi$, finding the correct parameters for the corresponding normalized Schwarz--Christoffel transformation is known as the \emph{parameter problem}, which has been studied extensively (see e.g. \cite{DriscollTrefethen}).

\section{The homeomorphism}

We now use the above properties of Schwarz--Christoffel transformations to prove \thmref{thm:main}. Let
\[
\calZ := \left\{ (z_3,\ldots,z_{n-1}) \in \RR^{n-3} : 0 < z_3 < \cdots < z_{n-1}  \right\}
\]
and
\[
\calA := \left\{ (\alpha_1, \ldots , \alpha_n) \in (0,2)^n : \sum_{j=1}^n \alpha_j = n-2 \right\}.
\]
We define a map 
\[
\Phi : \calZ \times \calA \to \LIP(n)
\] as follows. Given $(z_3,\ldots,z_{n-1}) \in \calZ$ and $(\alpha_1, \ldots , \alpha_n)\in \calA$, we add the prevertices $z_1=-1$, $z_2=0$, $z_n = \infty$, and let $F$ be the corresponding Schwarz--Christoffel transformation with remaining parameters $A=1$ and $B=0$. We then set 
\[
\Phi((z_3,\ldots,z_{n-1}),(\alpha_1, \ldots , \alpha_n)):= F|_{\RR \cup \{\infty \}}.
\]
By the continuous dependence of Schwarz--Christoffel transformations upon parameters, the map $\Phi$ is continuous.

We can also define a map
\[
\Psi : \LIP(n) \to \calZ \times \calA
\]
in the other direction as follows. For an immersed $n$-gon $\phi \in \LIP(n)$, let 
\[ z_1=-1 < z_2 = 0 < z_3 < \cdots < z_{n-1} < z_n = \infty, \]
$\alpha_1, \ldots, \alpha_n \in (0,2)$, and $A\in \CC\setminus\{0\}$, $B \in \CC$ be the unique parameters such that the boundary values of the corresponding Schwarz--Christoffel transformation coincide with $\phi$ up to reparametrization. We then set
\[
\Psi(\phi) = ((z_3,\ldots,z_{n-1}),(\alpha_1,\ldots,\alpha_n)),
\]
which is continuous.

By definition, it is clear that $\Psi \circ \Phi = \id$. However, $\Phi \circ \Psi$ might not give back the original immersed polygon because of our specific choice of $A$ and $B$. On the other hand, this problem goes away if we quotient out by $\Aff(\CC)$.

Indeed, observe that if we post-compose an immersed polygon $\phi$ with an orientation-preserving similarity $h$, then the normalized Schwarz-Christoffel transformation for $h\circ \phi$ is $h \circ F$, which has the same parameters as $F$ except for the constants $A$ and $B$. In other words, $\Psi(h\circ \phi)=\Psi(\phi)$, so $\Psi$ descends to a function (which we denote by $\wtilde{\Psi}$) on the quotient space $\LIP(n)/\Aff(\CC)$. Let us also denote by $\wtilde{\Phi}$ the map $\Phi$ post-composed with the quotient map $\LIP(n) \to \LIP(n)/\Aff(\CC)$.

We evidently still have $\wtilde{\Psi} \circ \wtilde{\Phi} = \id$ and we now have $\wtilde{\Phi}\circ \wtilde{\Psi} = \id$ as well. Indeed, if $\phi \in \LIP(n)$ is the restriction of a normalized Schwarz--Christoffel transformation $F$ with constants $A$ and $B$ to $\RR\cup \{ \infty \}$, then by definition
\[
\Phi(\Psi(\phi)) = (A^{-1}(F-B))|_{\RR \cup \{ \infty \}} = A^{-1}(\phi-B) \in \Aff(\CC) \circ \phi
\]
and hence $\wtilde{\Phi}(\wtilde{\Psi}([\phi])) = \wtilde{\Phi}(\Psi(\phi)) = [\phi]$ in $\LIP(n)/\Aff(\CC)$.

We have established that $\wtilde{\Psi} : \LIP(n)/\Aff(\CC) \to \calZ \times \calA$ is a continuous bijection with continuous inverse. It only remains to show that $\calZ \times \calA$ is homeomorphic to $\RR^{2n-4}$, which is elementary.

Firstly, the map
\[
\begin{array}{rcl}
\calZ & \to & \RR^{n-3} \\
(z_3,z_4,\ldots,z_{n-1}) & \mapsto & \left( \log(z_3), \log(z_4-z_3), \ldots, \log(z_{n-1}-z_{n-2}) \right)
\end{array}
\]
is clearly a homeomorphism.

Secondly, observe that $\calA$ is a convex relatively open set in the hyperplane $\calH \subset \RR^n$ defined by the equation $\sum_{j=1}^n \alpha_j = n-2$. We have that $\calH \cong \RR^{n-1}$ and it is well-known that any convex open set $U$ in $\RR^m$ is homeomorphic to $\RR^m$ itself. Indeed, one may define a form of polar coordinates with respect to any point in $U$. It follows that $\calA$ is homeomorphic to $\RR^{n-1}$, and hence
\[
\calZ \times \calA \cong \RR^{n-3} \times \RR^{n-1} = \RR^{2n-4}.
\]
This completes the proof of \thmref{thm:main}.

\section{Embedded polygons}

In this section, we show that when $n$ is at most $5$, an immersed $n$-gon is automatically embedded.

\begin{lem} \label{lem:immersed_is_embedded}
If $3 \leq n \leq 5$, then any immersed $n$-gon in the plane is simple.
\end{lem}
\begin{proof}
Let $\phi$ be an immersed $n$-gon which is not simple. We can perturb $\phi$ to a nearby non-simple immersed polygon $\psi$ where all self-intersections happen away from the vertices and are transverse. Since $\psi$ is not simple, there is some point $p$ in the complement of its image such that the winding number of $\psi$ around $p$ is at least $2$. Indeed, winding numbers are non-negative for immersed polygons and the winding number always takes 3 distinct values around a transverse intersection point. We can further assume that $p$ does not lie
on any of the lines that contain the sides of $\psi$.

Observe that each side of $\psi$ contributes strictly less than $1/2$ to the win\-ding number around $p$ because it spans an angle strictly less than $\pi$ as seen from $p$. It follows that $n > 4$.  

Suppose that $n=5$. Then each side of $\psi$ must be oriented to the left with respect to $p$, otherwise the winding number would be strictly less than $2=4\cdot \frac12$ due to the negative contributions. Let us form triangles with $p$ and each side of $\psi$. Then the sum of the angles $\gamma_k$ of these triangles at the point $p$ is equal to $2\pi k$ where $k\geq 2$ is the winding number of $\psi$ around $p$. On the other hand, by adding all the interior angles of these triangles, we find that
\[
\sum_{k=1}^n \gamma_k = \sum_{j=1}^n \pi - \theta_j
\]
where the $\theta_j$ are the interior angles of the polygon $\psi$. This last sum is the sum of the counter-clockwise turning angles of $\psi$ at the vertices, which always equals $2\pi$ for an immersed polygon. Indeed, the sum of the interior angles of an immersed polygon is $(n-2)\pi$ by \eqnref{eqn:angles} and \thmref{thm:SC}. This is a contradiction, from which we conclude that $n \geq 6$.
\end{proof}

Observe that the statement of the lemma is false for $n=6$, as illustrated in \figref{fig:hex}. Also note that the Schwarz--Christoffel formula makes sense if some of the ``interior angles'' $\theta_j = \alpha_j \pi$ are allowed to be larger than or equal to $2\pi$, but in that case the resulting map is no longer locally injective at the corresponding prevertices. \figref{fig:pentagon} shows that the analogue of \lemref{lem:immersed_is_embedded} for such maps fails for $n=5$, so the restriction $\alpha_j \in (0,2)$ is crucial for our purposes. The part of the proof that breaks down without this restriction is that at a vertex $w_j$ with ``interior angle'' $\theta_j > 2\pi$, the turning angle is $3\pi - \theta_j$ rather than $\pi - \theta_j$.

\begin{figure}[htp]
\centering
\includegraphics[width=0.4\textwidth]{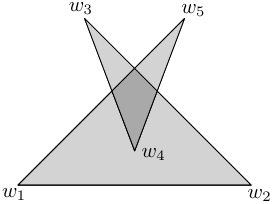}
\caption{A pentagon which fails to be immersed at one vertex.} \label{fig:pentagon}
\end{figure}

We conclude by observing that \corref{cor:emb} follows immediately from \thmref{thm:main} and \lemref{lem:immersed_is_embedded}.

\bibliographystyle{amsalpha}
\bibliography{biblio}

\end{document}